\documentclass[11pt]{amsart}

\newtheorem{theorem}{Theorem}[section]
\newtheorem{lemma}[theorem]{Lemma}

\theoremstyle{definition}

\theoremstyle{remark}

\numberwithin{equation}{section}

\begin{document}

\title[Domination by positive weak$^{*}$ Dunford-Pettis operators]
 {Domination by positive weak$^{*}$ Dunford-Pettis operators on Banach lattices }
\author[J.X. Chen]
{Jin Xi Chen }

\address{College of Mathematics and Information Science, Shaanxi Normal
University, Xi'an 710062, P.R. China}
\address{Department of Mathematics, Southwest Jiaotong
University, Chengdu 610031, P.R. China}
 \email{jinxichen@home.swjtu.edu.cn}

\author[Z.L. Chen]
{Zi Li Chen}
\address{Department of Mathematics, Southwest Jiaotong
University, Chengdu 610031, P.R. China}
\email{zlchen@home.swjtu.edu.cn}

\author[G.X. Ji]
{Guo Xing Ji}
\address{College of Mathematics and Information Science, Shaanxi Normal
University, Xi'an 710062, P.R. China}
\email{gxji@snnu.edu.cn}

\thanks{The first author was supported in part by NSFC (No.11301285) and the Fundamental Research Funds for the Central Universities (SWJTU11CX154). The second author was supported in part by the Fundamental Research Funds for the Central Universities (SWJTU12ZT13). The third author was supported in part by NSFC (No.11371233).}

%    General info
%\dedicatory{}


\begin{abstract}
Recently, J. H'michane et al. introduced the class of weak$^*$ Dunford-Pettis operators on Banach spaces, that is, operators which send weakly compact sets onto limited sets. In this paper the domination problem for weak$^*$ Dunford-Pettis operators is considered. Let $S, T:E\rightarrow F$ be two positive operators between Banach lattices $E$ and $F$ such that $0\leq S\leq T$. We show that if $T$ is a weak$^{*}$ Dunford-Pettis operator and $F$ is $\sigma$-Dedekind complete, then $S$ itself is weak$^*$ Dunford-Pettis.
\end{abstract}

\maketitle \baselineskip 5mm

\section{Introduction}
\par Throughout this paper $X,\,Y$ will denote  Banach spaces, and $E,\,F$ will denote Banach lattices. $sol(A)$ denotes the solid hull of a subset $A$ of a Banach lattice. The positive cone of $E$ will be denoted by $E^{\,+}$.   Following K. T. Andrews \cite{And} (or J. Bourgain and J. Diestel \cite{BD}) we say that a norm bounded subset $A$ of $X$ is  a \textit{Dunford-Pettis set} (resp. a \textit{limited set}) whenever every weakly null sequence in $X^{*}$ (resp. weak$^{*}$ null sequence in $X^{*}$) converges uniformly to zero  on $A$. Clearly, every relatively compact set in $X$ is a limited set, and every limited set in $X$ is a Dunford-Pettis set, but the converses are not true in general. Let us recall that a linear operator $T:X\rightarrow Y$ is called to be a \textit{Dunford-Pettis operator} if $x_{n}\xrightarrow {w} 0$ in $X$ implies $\|Tx_n\|\rightarrow0$, equivalently, if $T$ carries relatively weakly compact subsets of $X$ onto relatively compact subsets of $Y$. Aliprantis and Burkinshaw \cite{AB2} introduced a class of operators related to the Dunford-Pettis operators,  the so-called weak Dunford-Pettis operators. A bounded linear operator $T:X\rightarrow Y$ between Banach spaces is said to be a \textit{weak Dunford-Pettis operator} whenever $x_{n}\xrightarrow {w} 0$ in $X$ and $f_{n}\xrightarrow {w} 0$ in $Y^{*}$ imply $\lim_{\,n} f_{n}(Tx_{n})=0$, or equivalently, whenever $T$ carries relatively weakly compact subsets of $X$ onto Dunford-Pettis subsets of $Y$.
\par Recently, H'michane et al. \cite{HKBM} introduced the class of weak$^*$ Dunford-Pettis operators, and characterized this class of operators  and studied some of its properties in \cite{KHBM}. Following H'michane et al. \cite{HKBM} we say a bounded linear operator $T:X\rightarrow Y$  is a \textit{weak$^*$ Dunford-Pettis operator} whenever $x_{n}\xrightarrow {w} 0$ in $X$ and $f_{n}\xrightarrow {w^{*}} 0$ in $Y^{*}$ imply $ f_{n}(Tx_{n})\rightarrow0$, or equivalently, whenever $T$ carries relatively weakly compact subsets of $X$ onto limited subsets of $Y$ (\cite[Theorem 3.2]{KHBM}).

\par Recall that in the literature the domination problem for a class  $\mathcal{C}$ of
operators acting between Banach lattices is stated as follows:
\begin{itemize}
  \item Let $S, T:E\rightarrow F$ be two positive operators between Banach lattices such that $0 \leq S\leq T$. Assume
that $T$ belongs to the class $\mathcal{C}$. Which conditions on $E$ and $F$ do ensure that
$S$ belongs to $\mathcal{C}$?
\end{itemize}
In \cite{Kalton} Kalton and Saab established that a positive operator from $E$ into $F$, dominated by a positive weak Dunford-Pettis operator, must also be weak Dunford-Pettis, whilst they obtained such a result for Dunford-Pettis operators provided the norm of $F$ is order continuous. Later, Wickstead \cite{Wickstead} studied the converse for the Kalton-Saab theorem: every positive operator from $E$ into $F$ dominated by a Dunford-Pettis operator is Dunford-Pettis if and only if $E$ has weakly sequentially continuous lattice operations or $F$ has order continuous norm.

\par Naturally, we come to the case of weak$^*$ Dunford-Pettis operators. The main purpose of this paper is to study the domination problem for positive weak$^*$ Dunford-Pettis operators between Banach lattices.
Let $S:E\rightarrow F$ be a positive operator between Banach lattices $E$ and $F$ such that $F$ is $\sigma$-Dedekind complete. We show that if $S$ is dominated by a positive weak$^{*}$ Dunford-Pettis operator, then $S$ itself is weak$^*$ Dunford-Pettis.

\par Our notions are standard. For the theory of Banach lattices and operators, we refer  the reader to the monographs \cite{AB, M}.

\section{Lattice Properties of Positive Weak$^{*}$ Dunford-Pettis Operators}

\par It should be noted that in a Banach lattice (or in its dual) the lattice operations fail to be weakly (\,resp. weak$^{*}$) sequentially continuous  in general. Let us recall that every disjoint sequence in the solid hull of a relatively weakly compact subset of a Banach lattice $E$ converges weakly to zero (cf. \cite[Theorem 4.34]{AB}). In particular, if $(x_{n})$ is a disjoint, weakly convergent sequence in $E$, then the sequences $(x_{n})$, $(|\,x_{n}|)$, $(x_{n}^{+})$, $(x_{n}^{-})$ all converge weakly to zero. However, from  Example 2.1 of \cite{CCJ} we can see such a property need not be possessed by $w^{*}$-convergent disjoint sequences in the dual space.

\par The following lemma, which deals with disjoint sequences in the dual of a  $\sigma$-Dedekind complete Banach lattice, is due to the authors \cite{CCJ} and is needed in the rest of this paper.
\begin{lemma}[\cite{CCJ}]\label{Lemma 1}
Let $E$ be a $\sigma$-Dedekind complete Banach lattice, and let $(f_{n})$ be a $w^{\,*}$-convergent sequence of $E^{\,*}$. If $(g_n)$ is  a disjoint sequence of $E^{\,*}$ satisfying $|\,g_n|\leq |\,f_n|$ for each $n\in \mathbb{N}$, then the sequences $(g_{n})$, $(|\,g_{n}|)$, $(g_{n}^{\,+})$, $(g_{n}^{\,-})$ all weak$^{\,*}$ converge to zero. In particular, if $(f_n)$ is a disjoint $w^*$-convergent sequence in its own right, then the sequences $(f_{n})$, $(|\,f_{n}|)$, $(f_{n}^{\,+})$, $(f_{n}^{\,-})$ are all weak$^*$ null.
\end{lemma}

\par Let $T:E\rightarrow F$ be a positive weak$^*$ Dunford-Pettis operator between Banach lattices. For every  weakly null sequence $(x_n)$ in $E^{\,+}$ and every weak$^*$ null sequence $(f_n)$  in $F^{\,*}$, by the definition of weak$^*$ Dunford-Pettis operators  we have $f_n(Tx_n)\rightarrow 0$. Indeed we can say more when $F$ is $\sigma$-Dedekind complete.

\begin{theorem}\label{Theorem 2}
Let $T:E\rightarrow F$ be a positive weak$^*$ Dunford-Pettis operator between Banach lattices $E$  and $F$ with $F$ $\sigma$-Dedekind complete.
 Then for every weakly null sequence $(x_n)$ in $E^{\,+}$ and every weak$^*$ null sequence $(f_n)$  in $F^{\,*}$ we have $|f_n|(Tx_n)\rightarrow 0\quad (n\rightarrow\infty).$
\end{theorem}

\begin{proof}
Let $\varepsilon>0$ be arbitrary. First, we claim that there exists $0\leq g\in F^{*}$ and $N\in \mathbb{N}$ such that
$$ (|f_n|-g)^{+}(Tx_n)<\varepsilon\eqno{(\ast)}$$ holds for all $n>N$. Suppose that $(\ast)$ is false. Then there exists an $\varepsilon^{\,\prime}>0$ such that for each
$0\leq g\in F^{*}$ and each $N\in \mathbb{N}$ we have $ (|f_k|-g)^{+}(Tx_k)\geq\varepsilon^{\,\prime}$ for at least one  $k>N$. Let us put $g=4|f_1|$ and $n_1=1$. Thus there exists a natural number $n_2 \,(>n_1)$ satisfying $$(|f_{n_2}|-4|f_1|)^{+}(Tx_{n_2})\geq\varepsilon^{\,\prime}.$$Also, let us put $g=4^{2}\sum_{i=1}^{2}|f_{n_{i}}|$. Then, we have $$\left(|f_{n_{3}}|-4^{2}\sum_{i=1}^{2}|f_{n_{i}}|\right)^{+}(Tx_{n_{3}})\geq\varepsilon^{\,\prime}$$ for some natural number $n_3 \,(>n_2).$ Proceeding with an inductive argument we can obtain a strictly increasing subsequence $(n_k)$ of $\mathbb{N}$ such that $$\left(|f_{n_{k+1}}|-4^{k}\sum_{i=1}^{k}|f_{n_{i}}|\right)^{+}(Tx_{n_{k+1}})\geq\varepsilon^{\,\prime}$$for all $k\in \mathbb{N}$. Let $f=\sum_{k=1}^{\infty}2^{-k}|f_{n_{k}}|$, and put
$$g_{k+1}=\left(|f_{n_{k+1}}|-4^{k}\sum_{i=1}^{k}|f_{n_{i}}|\right)^{+},\qquad \tilde{f}_{k+1}=\left(|f_{n_{k+1}}|-4^{k}\sum_{i=1}^{k}|f_{n_{i}}|-2^{-k}f\right)^{+}.$$ Note that
$0\leq g_{k+1}\leq \tilde{f}_{k+1}+2^{-k}f$ and $g_{k+1}(Tx_{n_{k+1}})\geq\varepsilon^{\,\prime}$ for any $k\in \mathbb{N}$. By Lemma 4.35 of \cite{AB} $(\tilde{f}_{k+1})$ is a disjoint sequence. Since $0\leq \tilde{f}_{k+1}\leq|f_{n_{k+1}}|$ and $f_{n_{k+1}}\xrightarrow {w^{*}} 0$, in view of Lemma \ref{Lemma 1} we have $\tilde{f}_{k+1}\xrightarrow {w^{*}} 0$ in $F^{\,*}$. From the weak$^*$ Dunford-Pettis property of $T$ it follows that $\tilde{f}_{k+1}(Tx_{n_{k+1}})\rightarrow0$. However,
\begin{eqnarray*}
0<\varepsilon^{\,\prime}\leq g_{k+1}(Tx_{n_{k+1}})&\leq &(\tilde{f}_{k+1}+2^{-k}f)(Tx_{n_{k+1}})\\&=&\tilde{f}_{k+1}(Tx_{n_{k+1}})+2^{-k}f(Tx_{n_{k+1}})\rightarrow0.
\end{eqnarray*}This leads to a contradiction. Hence, $(\ast)$ is true.
\par Now Let $0\leq g\in F^{*}$ and $N\in \mathbb{N}$ satisfy $(\ast)$. For all $n>N$, we have the inequalities
\begin{eqnarray*}
 |f_n|(Tx_n)=(|f_n|-g)^{+}(Tx_n)+(|f_n|\wedge g)(Tx_n)&\leq&(|f_n|-g)^{+}(Tx_n)+ g(Tx_n)\\&\leq&\varepsilon+g(Tx_n).
\end{eqnarray*}Because $x_{n}\xrightarrow {w} 0$ in $E$, it follows that $\limsup|f_n|(Tx_n)\leq\varepsilon$. Since $\varepsilon>0$ is arbitrary, we have
$\lim_{n}|f_n|(Tx_n)\rightarrow 0,$ as desired.
\end{proof}

\par The next theorem describes an important approximation property of positive weak$^*$ Dunford-Pettis operators.
\begin{theorem}\label{Theorem 3}
Let $T:E\rightarrow F$ be a positive weak$^*$ Dunford-Pettis operator between Banach lattices $E$  and $F$ with $F$ $\sigma$-Dedekind complete.
Let $W$ be a relatively weakly compact subset of $E$ and $(f_n)$ be a weak$^*$ null sequence in $F^*$. Then, for any $\varepsilon>0$ there exists some $N\in \mathbb{N}$ and some $u\in E^{+}$ lying in the ideal generated by $W$ such that $$|f_n|\left(T(|x|-u)^{+}\right)<\varepsilon$$ for all $n>N$ and all $x\in W$.
\end{theorem}
\begin{proof}
Assume by way of contradiction that the claim is false. That is, there exists an $\varepsilon^{\,\prime}>0$ such that for each $N\in \mathbb{N}$ and each $u\geq0$ in the ideal generated by $W$ we can find a natural number $m>N$ and some $x_m\in W$ satisfying $|f_m|(T(|x_m|-u)^{+})\geq\varepsilon^{\,\prime}$.
Hence, by an easy inductive argument we can choose a strictly increasing subsequence $(n_k)$ of $\mathbb{N}$ and a sequence $(x_k)\subseteq W$ such that
$$|f_{n_{k+1}}|\left(T\left(|x_{k+1}|-4^{k}\sum_{i=1}^{k}|x_i|\right)^{+}\right)\geq\varepsilon^{\,\prime}>0\eqno{(\ast\ast)}$$Let $x=\sum_{k=1}^{\infty}2^{-k}|x_{k}|$. Also put$$w_{k+1}=\left(|x_{k+1}|-4^{k}\sum_{i=1}^{k}|x_i|\right)^{+},\qquad v_{k+1}=\left(|x_{k+1}|-4^{k}\sum_{i=1}^{k}|x_i|-2^{-k}x\right)^{+}. $$Clearly, $|f_{n_{k+1}}|(Tw_{k+1})\geq\varepsilon^{\,\prime}>0$ and $0\leq w_{k+1}\leq v_{k+1}+2^{-k}x$ hold for all $k\in\mathbb{ N}$. By Lemma 4.35 of \cite{AB} $(v_{k+1})$ is a disjoint sequence in $E$. Note that $0\leq v_{k+1}\leq |x_{k+1}|$. It follows that $v_{k+1}\in sol(W)$ holds for all $k$. Since every disjoint sequence in the solid hull of a relatively weakly compact set of a Banach lattice
converges weakly to zero (cf. \cite[Theorem 4.34]{AB}), we see that $v_{k+1}\xrightarrow {w} 0$ in $E$. Hence, from Theorem \ref{Theorem 2} it follows that $|f_{n_{k+1}}|(Tv_{k+1})\rightarrow0.$ On the other hand,
\begin{eqnarray*}
 0<\varepsilon^{\,\prime}\leq|f_{n_{k+1}}|(Tw_{k+1})&\leq&|f_{n_{k+1}}|\left(T(v_{k+1}+2^{-k}x)\right)
 \\&=&|f_{n_{k+1}}|(Tv_{k+1})+2^{-k}|f_{n_{k+1}}|(Tx)\rightarrow0,
\end{eqnarray*}which contradicts $(\ast\ast)$. Therefore, the proof is completed.
\end{proof}
\section{Domination by Positive Weak$^{*}$ Dunford-Pettis Operators}
\par Let us recall that a Banach space $X$ is said to be a \textit{Gelfand-Phillips space} whenever all limited sets in $X$ are relatively compact. It is well known that all separable Banach spaces and all weakly compactly generated spaces are Gelfand-Phillips spaces. Note that a $\sigma$-Dedekind complete Banach lattice $E$ is a Gelfand-Phillips space if and only if the norm of $E$ is order continuous (cf. \cite{B}).  $X$ has the\textit{ Dunford-Pettis property }(resp. the \textit{Dunford-Pettis$^{*}$ property})  whenever every relatively weakly compact set in $X$ is a Dunford-Pettis set (resp. a limited set), in other words, for each weakly null sequence $(x_{n})$ in $X$ and each weakly null sequence (resp. weak$^{*}$ null sequence) $(f_{n})$ in $X^{*}$, $\lim_{\,n} f_{n}(x_{n})=0$. The Dunford-Pettis$^{*}$ property, introduced first by Borwein, Fabian and Vanderwerff \cite{BFV}, is stronger than the Dunford-Pettis property. Carri\'{o}n, Galindo and Louren\c{c}o \cite{CGL} showed that $X$ has the Dunford-Pettis$^{*}$ property if, and only if, every bounded linear operator $T:X\rightarrow c_{0}$ is a Dunford-Pettis operator.
\par It should be noted that if either $X$ or $Y$ has the Dunford-Pettis$^{*}$ property, then every bounded linear operator from $X$ into $Y$ is weak$^*$ Dunford-Pettis. Also, If $Y$ is a Gelfand-Phillips space,  weak$^*$ Dunford-Pettis operators from $X$ into $Y$  and Dunford-Pettis operators between them coincide.
\par Now, since $L_{1}[0,\,1]$ does not have weakly sequentially continuous lattice operations and the norm of $c$ is not order continuous, by the converse for the Kalton-Saab theorem proved by Wickstead \cite{Wickstead} we know that there exists a positive operator from $L_{1}[0,\,1]$ into $c$, which is dominated by a Dunford-Pettis operator, is not Dunford-Pettis. On the other hand, since $c$ is a Gelfand-Phillips space, every weak$^*$ Dunford-Pettis operators from $L_{1}[0,\,1]$ into $c$ is Dunford-Pettis. Therefore,  there exists a positive operator from $L_{1}[0,\,1]$ into $c$ dominated by a weak$^{*}$ Dunford-Pettis operator is not weak$^{*}$ Dunford-Pettis. It should be noted that $c$ is not $\sigma$-Dedekind complete. In case the range space is $\sigma$-Dedekind complete, we have the following  domination result for positive weak$^*$ Dunford-Pettis operators.
\begin{theorem}\label{Theorem 4}
Let $E$, $F$  be two Banach lattices such that $F$ is $\sigma$-Dedekind complete. If a positive operator $S:E\rightarrow F$ is dominated by
 a positive weak$^*$ Dunford-Pettis operator, then $S$ itself is weak$^*$ Dunford-Pettis.
\end{theorem}
\begin{proof}
We shall follow the plan of N. J. Kalton and P. Saab in their weak Dunford-Pettis version, but we have to make efforts to overcome some obstacles on our way since the behavior of weak$^*$ convergence is quite different from that of weak convergence in general.
\par Assume that $F$ is $\sigma$-Dedekind complete and that $T:E\rightarrow F$ is a positive weak$^*$ Dunford-Pettis operator satisfying $0\leq S\leq T$.
Let $x_n\xrightarrow {w}0$ in $E$, and let $f_{n}\xrightarrow {w^{*}} 0$ in $F^{*}$. To prove that $S$ is weak$^*$ Dunford-Pettis, we have to show that $f_{n}(Sx_n)\rightarrow0.$ To this end, put $e=\sum_{n=1}^{\infty}2^{-n}|x_n|\in E^+$, and let $A_e$ be the ideal generated in $E$ by $e$. Consider the operators $0\leq S \leq T:\bar{A_e}\rightarrow F$, where $\bar{A_e}$ is the norm closure of $A_e$ in $E$.
Clearly, $x_n\xrightarrow {w}0$ in $\bar{A_e}$,  and $T:\bar{A_e}\rightarrow F$ is likewise weak$^*$ Dunford-Pettis. Let $\varepsilon>0$ be fixed. By Theorem \ref{Theorem 3} there exist some $N_1\in \mathbb{N}$ and some $u\in E^{+}$ lying in the ideal generated by $(x_n)$ such that $$|f_n|\left(T(|x_n|-u)^{+}\right)<\varepsilon$$ for all $n>N_1$. Note that $u\in A_{e}$. Since $F$ is $\sigma$-Dedekind complete and $f_{n}\xrightarrow {w^{*}} 0$ in $F^{*}$,  there exists $0\leq g \in F^{\,*}$ lying in the ideal generated by $(f_{n})$ in $F^*$ such that $$(|\,f_n|-g)^{+}(Tu)<\varepsilon$$holds for all $n\in\mathbb{N}$ (\,\cite{Burk}; cf. \cite[Theorem 4.42]{AB}).
\par On the other hand,
By Theorem 4.82 of \cite{AB} there exist positive operators $M_1,\cdot\cdot\cdot,M_k$ on $\bar{A_e}$ and order projections $P_1,\cdot\cdot\cdot,P_k$ on $F^{**}$ satisfying $$\left<g,\,\left|S-\sum_{i=1}^{k}P_{i}TM_i\right|u\right><\varepsilon \qquad \textrm{and} \qquad 0\leq\sum_{i=1}^{k}P_{i}TM_i\leq T.$$(The proof of the extension of each positive multiplication operator $M_i$ on $A_e$ to a positive operator on $\bar{A_e}$ can be found in Part(b) on p.269 in \cite{AB}.) Let us put $R=|S-\sum_{i=1}^{k}P_{i}TM_i|$. Obviously, $$\left<g, Ru\right><\varepsilon\qquad \textrm{and} \qquad R=\left|S-\sum_{i=1}^{k}P_{i}TM_i\right|\leq S+\sum_{i=1}^{k}P_{i}TM_i\leq2T.$$ For each $n>N_1$, we have
\begin{eqnarray*}
\left<|f_{n}|, R|x_{n}|\right>&=&\left<|f_{n}|, R(|x_{n}|-u)^+\right>+\left<|f_{n}|, R(|x_{n}|\wedge u)\right>\\&\leq&\left<|f_{n}|, R(|x_{n}|-u)^+\right>+\left<|f_{n}|, Ru\right>\\&\leq&2|f_{n}|(T(|x_{n}|-u)^+)+\left<|f_{n}|, Ru\right>\\&\leq&2|f_{n}|(T(|x_{n}|-u)^+)+\left<(|f_{n}|-g)^+, Ru\right>+\left<g, Ru\right>\\&\leq&2|f_{n}|(T(|x_{n}|-u)^+)+2(|f_{n}|-g)^+(Tu)+\left<g, Ru\right>\\&<&2\varepsilon+2\varepsilon+\varepsilon=5\varepsilon.
\end{eqnarray*}This implies
\begin{eqnarray*}
 |f_n(Sx_{n})|&\leq&\left|\left<f_n, \left(S-\sum_{i=1}^{k}P_{i}TM_i\right)x_n\right>\right|+\sum_{i=1}^{k}\left|\left<f_n, P_{i}TM_{i}x_n\right>\right|
\\&\leq&\left<|f_n|, \left|S-\sum_{i=1}^{k}P_{i}TM_i\right||x_n|\right>+\sum_{i=1}^{k}\left|\left<f_n, P_{i}TM_{i}x_n\right>\right|\\&=&\left<|f_{n}|, R|x_{n}|\right>+\sum_{i=1}^{k}\left|\left<f_n, P_{i}TM_{i}x_{n}\right>\right|\\&<&5\varepsilon+\sum_{i=1}^{k}\left|\left<f_n, P_{i}TM_{i}x_{n}\right>\right|
\end{eqnarray*}holds for all $n>N_1$. To prove that $\lim_{n}f_n(Sx_{n})=0$, we need only to show that $\left<f_n, P_{i}TM_{i}x_{n}\right>\rightarrow0\quad (n\rightarrow\infty)$ for $i=1, 2,\cdot\cdot\cdot,k.$  Note that each $P_i$ is an order projection on $F^{\ast\ast}$. For each $f\in F^{*}$ we define $Q_if$ by $$(Q_{i}f)y=\left<f, P_{i}y\right>,\qquad \forall\,y\in F.$$We can easily see that $Q_{i}f\in F^{*}$ and $Q_{i}:F^{*}\rightarrow F^{*}$ is a bounded linear operator. Now we claim that $Q_{i}$ is sequentially $w^{*}$-continuous on $F^*$. If this claim is true, then $f_{n}\xrightarrow {w^{*}} 0$ in $F^{*}$ implies that $Q_{i}f_{n}\xrightarrow {w^{*}} 0$ in $F^{*}$. It turns out that  by the weak$^*$ Dunford-Pettis property we have$$\left<f_n, P_{i}TM_{i}x_{n}\right>=\left<TM_{i}x_n, Q_{i}f_{n}\right>\rightarrow0,\quad(n\rightarrow\infty)$$
since $M_{i}x_{n}\xrightarrow {w}0\quad(n\rightarrow\infty)$ in $E$ $(\textrm{or}\, \bar{A_e})$.
\par So, the key point is to prove the claim that each $Q_{i}$ defined above is sequentially $w^{*}$-continuous on $F^*$. To this end, assume $f_n\xrightarrow {w^{*}}0$ in $F^{*}$, and let $A_{F}$ denote the ideal generated by $F$ in $F^{\ast\ast}$. Then $f_n\xrightarrow {\sigma(F^{*},\,A_{F})}0$ in $F^{\ast}$ since $F$ is $\sigma$-Dedekind complete (cf. \cite[Theorem 4.43]{AB}). Let us recall that $P_{i}$ is an order projection on $F^{\ast\ast}$. Given $y\in F^{+}$. Since $0\leq P_{i}y\leq y$, we can see that $P_{i}y\in A_{F}$. Hence,$$\left<y, Q_{i}f_{n}\right>=\left<f_{n}, P_{i}y\right>\rightarrow0\quad(n\rightarrow\infty),$$which implies the sequential $w^{*}$-continuity of $Q_i$. The proof is finished.
\end{proof}
\vskip 5mm


\begin{thebibliography}{99}

\bibitem{AB2}C.D. Aliprantis and O. Burkinshaw, Dunford-Pettis operators on Banach lattices, \textit{Trans. Amer. Math. Soc.} \textbf{274} (1982), 227-238.
\bibitem{AB} C.D. Aliprantis and O. Burkinshaw, \textit{Positive
Operators} (reprint of the 1985 original), Springer, Dordrecht, 2006.

\bibitem{And}K.T. Andrews, Dunford-Pettis sets in the space of Bochner integrable functions,\textit{ Math. Ann.} \textbf{241} (1979), 35-41.
 \bibitem{BFV} J. Borwein, M. Fabian and J. Vanderwerff, Characterizations of Banach spaces via convex and other locally Lipschitz functions, \textit{Acta Math. Vietnam} \textbf{22} (1997), 53-69.
 \bibitem{BD}J. Bourgain and J. Diestel, Limited operators and strict cosingularity, \textit{Math. Nachr.} \textbf{119} (1984), 55-58.
 \bibitem{B}A.V. Buhvalov, Locally convex spaces that are generated by weakly compact sets (Russian), \textit{Vestnik Leningrad Univ. No.7 Mat. Meh. Astronom. Vyp.}\textbf{ 2} (1973), 11-17.

 \bibitem{Burk}O. Burkinshaw, Weak compactness in the order dual of a vector lattice, \textit{Trans. Amer. Math. Soc.} \textbf{187} (1974), 183-201.
 \bibitem{CGL}H. Carri\'{o}n, P. Galindo and M.L. Louren\c{c}o, A stronger Dunford-Pettis property, \textit{Studia Math.} \textbf{184} (2008), 205--216.
\bibitem{CCJ}J.X. Chen, Z.L. Chen and G.X. Ji, Almost limited sets in Banach lattices, \textit{J. Math. Anal. Appl.}(2014), http://dx.doi.org/10.1016/j.jmaa.2013.10.085


\bibitem{HKBM}J. H'michane, A. El Kaddouri, K. Bouras and M. Moussa, On the class of limited operators, \textit{Acta Math. Sci.}, to appear.
\bibitem{KHBM}A. El Kaddouri, J. H'michane, K. Bouras and M. Moussa, On the class of weak$^{*}$ Dunford-Pettis operators, \textit{Rend. Circ. Mat. Palermo (2)}\textbf{ 62 }(2013), 261-265.
\bibitem{Kalton}N.J. Kalton and P. Saab, Ideal properties of regular operators between Banach lattices, \textit{Illinois J. Math.}\textbf{ 29} (1985), 382-400.
\bibitem{M} P. Meyer-Nieberg, \textit{Banach Lattices}, Universitext, Springer-Verlag, Berlin, 1991.
\bibitem{Wickstead}A.W. Wickstead, Converses for the Dodds-Fremlin and Kalton-Saab theorems, \textit{Math. Proc. Camb. Phil. Soc.}  \textbf{120 }(1996), 175-179.

\end{thebibliography}
\end{document}